\documentclass[11pt]{amsart}

\usepackage[T1]{fontenc}
\usepackage{lmodern}
\usepackage{microtype}
\usepackage{amsmath,amsfonts,amssymb,amsthm,mathtools}
\usepackage{geometry}
\usepackage{setspace}
\usepackage{enumitem}
\usepackage{hyperref}

\hypersetup{
  colorlinks=true,
  linkcolor=blue,
  citecolor=blue,
  urlcolor=blue
}

\newcommand{\Sph}{\mathbb S}
\newcommand{\CP}{\mathbb{CP}}

\newcommand{\Area}{\operatorname{Area}}
\newcommand{\Ric}{\operatorname{Ric}}
\newcommand{\Conf}{\operatorname{Conf}}

\theoremstyle{plain}
\newtheorem{theorem}{Theorem}[section]
\newtheorem{corollary}[theorem]{Corollary}
\newtheorem{proposition}[theorem]{Proposition}
\newtheorem{lemma}[theorem]{Lemma}

\newtheorem{mainthm}{Theorem}

\theoremstyle{remark}
\newtheorem{remark}[theorem]{Remark}

\title[Second Jacobi Eigenvalues]{Second Jacobi Eigenvalues and Spectral Rigidity
for Surfaces in Berger Spheres}

\author{M\'arcio Batista}
\author{Abra\~ao Mendes}
\address{CPMAT -- Instituto de Matemática, Universidade Federal de Alagoas,
Macei\'o, AL, 57072-970, Brazil}
\email{mhbs@mat.ufal.br}
\email{abraao.mendes@im.ufal.br}


\numberwithin{equation}{section}

\usepackage{enumitem}

\usepackage{marginnote}

\begin{document}

\begin{abstract}
In this paper, we establish an upper bound for the second eigenvalue of the scalar Jacobi operator of an arbitrary closed two-sided surface immersed in a contracted Berger sphere.  The estimate involves the total squared mean curvature, the Euler characteristic, and an explicit nonpositive term determined by the angle between the surface normal and the Hopf direction.  The proof combines the realization of a Berger sphere as a geodesic hypersurface of a complex projective plane, the first standard embedding of the latter into a Euclidean sphere, a weighted Hersch--Li--Yau balancing argument, and the conformal invariance of the Willmore functional.  No minimality or constant mean curvature assumption is imposed.  As an application, if $1/3\leq\alpha\leq1$ and the surface has nonpositive Euler characteristic, then its second Jacobi eigenvalue is nonpositive; it is strictly negative for $\alpha>1/3$.  At the critical value $\alpha=1/3$, equality forces the immersed image to be congruent to the minimal Clifford torus; in the embedded category, this yields a complete characterization of the equality case.
\end{abstract}

\maketitle

\section{Introduction}

Let $\varphi:\Sigma^2\to M^3$ be a closed two-sided immersed surface,
with globally defined unit normal $N$.  Its scalar Jacobi operator is
\begin{equation*}\label{eq:intro-jacobi}
 \mathcal J=\Delta+|\sigma|^2+\Ric_M(N,N),
\end{equation*}
where $\Delta=\operatorname{div}\nabla$ and $\sigma$ is the second
fundamental form {of $\varphi$}.  We denote by
\[
 \lambda_1(\mathcal J)<\lambda_2(\mathcal J)\leq\lambda_3(\mathcal J)\leq\cdots \nearrow \infty
\]
the eigenvalues of $-\mathcal J$, repeated according to multiplicity.  The low
spectrum of $J$ contains information which is finer than the sign of the
second variation: it relates stability and Morse index to the intrinsic
topology and the extrinsic geometry of the immersion.

Sharp estimates for the second eigenvalue of Schr\"odinger and Jacobi
operators form a classical part of spectral submanifold geometry.  In
Euclidean space, Harrell and Loss proved that the second Jacobi eigenvalue
of a closed orientable hypersurface is nonpositive and characterized the
equality case by round spheres \cite{HarrellLoss1998}.  El Soufi and Ilias
developed a general conformal method for Schr\"odinger operators on
submanifolds of space forms \cite{ElSoufiIlias2000}.  Its essential
ingredients are coordinate test functions centered with respect to a first
eigenfunction and an estimate of their Dirichlet energy by an extrinsic
conformal invariant.  Related rigidity results for the Jacobi operator were
obtained by {the second-named author} \cite{Mendes2019}.  For surfaces, this circle of ideas has
recently produced sharp estimates in terms of Willmore energy and topology,
including spectral characterizations of the Clifford torus and other
extremal minimal surfaces \cite{BatistaCavalcanteMendesNunes2025}.

There is a basic difference between the first and second Jacobi 
eigenvalues. A first eigenfunction has {a fixed} sign{,} and coordinate functions
of a spherical immersion are not automatically orthogonal to it. To 
estimate the second eigenvalue, the auxiliary immersion must first be 
composed with a conformal transformation whose coordinates have zero 
weighted mean, the weight being the positive first eigenfunction. This 
centering principle goes back to Hersch \cite{Hersch1970} and Li--Yau 
\cite{LiYau1982}; in the Schr\"odinger setting{,} it is a central 
ingredient of {the work of} El Soufi and Ilias \cite{ElSoufiIlias2000}. For surfaces{,} it 
combines particularly well with the conformal invariance of {both} the Dirichlet 
energy and the Euclidean Willmore functional.

Berger spheres provide a natural non-space-form setting in which this
program can be tested.  They constitute a one-parameter family of
homogeneous metrics on $\Sph^3$, obtained by changing the metric along the
fibers of the Hopf fibration.  Their constant mean curvature surfaces have
been studied from several complementary viewpoints.  Rotational CMC
surfaces and compact minimal examples were investigated by Torralbo
\cite{Torralbo2010,Torralbo2012}.  Torralbo and Urbano obtained curvature
restrictions for compact surfaces in homogeneous three-manifolds
\cite{TorralboUrbano2010Gauss} and, most importantly for the present work,
classified compact orientable stable CMC surfaces in a substantial range of
Berger parameters \cite{TorralboUrbano2012}.  In particular, when
$1/3\leq\alpha<1$, the stable CMC examples are CMC spheres and, at the
endpoint $\alpha=1/3$, the minimal Clifford torus.  Their result also leads
to the solution of the isoperimetric problem in that range.  Subsequent work
has further analyzed the index of compact minimal submanifolds of Berger
spheres \cite{TorralboUrbano2022}.

More precisely, Torralbo--Urbano prove that an orientable compact stable
CMC surface in $\Sph^3_\alpha$, $1/3\leq\alpha<1$, is either a CMC sphere
or, at $\alpha=1/3$, the minimal Clifford torus.  Their notion is
volume-preserving stability: the Jacobi quadratic form is tested on
functions of zero integral, and the CMC hypothesis is essential for this
variational interpretation.  The admissible space for $\lambda_2(\mathcal J)$ is
instead the $L^2$-orthogonal complement of a first eigenfunction of $-\mathcal J$.
These two codimension-one spaces need not coincide.  Consequently, the
estimate below is not a reformulation of their stability theorem.

The extrinsic construction used by Torralbo and Urbano is crucial here.  A
contracted Berger sphere $\Sph^3_\alpha$, $0<\alpha<1$, is realized as a
geodesic hypersurface of $\CP^2(1-\alpha)$, and the latter admits its first
standard isometric embedding into $\mathbb R^8$.  Torralbo and Urbano used
this composition, together with stability and the CMC condition, to control
the topology of stable surfaces.  Our purpose is different: we use the same
extrinsic geometry to construct test functions for the second eigenvalue of
$-\mathcal J$.  The first standard embedding already takes values in a round
$7$-sphere.  It can therefore be balanced directly by a conformal
transformation of $\Sph^7$, with weight given by a first eigenfunction of
$-\mathcal J$.  This removes the volume constraint and makes the argument applicable
to every closed two-sided immersion, without assuming that $H$ is constant.

Our main result is the following.  We write $\nu=\langle N,\xi\rangle$, where
$\xi$ is the unit Hopf vector field.

\begin{mainthm}\label{thm:intro-main}
Let \(\varphi:\Sigma^2\to\Sph^3_\alpha\), \(0<\alpha<1\), be a closed two-sided immersion.  Then
\begin{align}
 \lambda_2(\mathcal J)\leq{}&
 \frac{(1+\alpha)(1-3\alpha)}{2\alpha}
 -\frac{2}{\Area(\Sigma)}\int_\Sigma H^2\mskip1mu dA
 +\frac{4\pi\chi(\Sigma)}{\Area(\Sigma)}\notag\\
 &-\frac{(1-\alpha)^2}{2\alpha\Area(\Sigma)}
 \int_\Sigma(2-\nu^2)\nu^2\mskip1mu dA.                 \label{eq:intro-main}
\end{align}
For $\alpha=1$, the limiting estimate is
\begin{equation*}\label{eq:intro-round}
 \lambda_2(\mathcal J)\leq
 -2-\frac{2}{\Area(\Sigma)}\int_\Sigma H^2\mskip1mu dA
 +\frac{4\pi\chi(\Sigma)}{\Area(\Sigma)}.
\end{equation*}
\end{mainthm}

The angular term in \eqref{eq:intro-main} is nonpositive and retains
information about the interaction between the surface and the Hopf
fibration.  In particular, the estimate detects the critical Berger metric
and yields the following rigidity statement.

\begin{mainthm}\label{thm:intro-rigidity}
Let $1/3\leq\alpha\leq1$, and let
$\varphi:\Sigma^2\to\Sph^3_\alpha$ be a closed, connected, two-sided
immersion with $\chi(\Sigma)\leq0$.  Then $\lambda_2(\mathcal J)\leq0$.  Moreover:
\begin{enumerate}[label={\rm (\arabic*)}]
 \item if $1/3<\alpha\leq1$, then $\lambda_2(\mathcal J)<0$;
 \item if $\alpha=1/3$ and $\lambda_2(\mathcal J)=0$, then the image of $\varphi$
 is congruent to the minimal Clifford torus and $\varphi$ factors through a
 finite covering of its standard embedding.  Conversely, the standard
 Clifford-torus embedding has $\lambda_2(\mathcal J)=0$.  In particular, among
 embedded surfaces, equality holds if and only if the surface is congruent
 to the Clifford torus.
\end{enumerate}
\end{mainthm}

Thus the novelty relative to \cite{TorralboUrbano2012} is spectral rather
than a reformulation of CMC stability.  Theorem~\ref{thm:intro-main} is an
explicit estimate for $\lambda_2(\mathcal J)$ valid for non-CMC immersions; it keeps
both a total-mean-curvature term and a Hopf-angle term. Theorem~\ref{thm:intro-rigidity}
then characterizes the exceptional Clifford torus
solely by the value of the second Jacobi eigenvalue within the entire class
of closed two-sided surfaces of nonpositive Euler characteristic.  For CMC
surfaces, the spectral conclusion is consistent with the stability
classification of Torralbo and Urbano, but neither stability nor the CMC
equation is an assumption in our results.

\textbf{The paper is organized as follows}.  Section~\ref{sec:berger} recalls the
curvature identities of the Berger sphere.  In Section~\ref{sec:auxiliary}{,}
we describe the projective and Euclidean embeddings and calculate the mean
curvature of the composed immersion.  Section~\ref{sec:conformal} gives the
weighted balancing and energy estimate.  The main eigenvalue bound is proved
in Section~\ref{sec:estimate}.  Finally, Section~\ref{sec:rigidity} proves
sharpness and rigidity at $\alpha=1/3$.

All manifolds and immersions are smooth.  Eigenvalues are repeated
according to multiplicity, and all integrals and differential operators on
$\Sigma$ are taken with respect to the metric induced by $\varphi$.

\section{Geometry of Berger spheres}\label{sec:berger}

Let
\(
 \Sph^3=\{(z,w)\in\mathbb C^2:|z|^2+|w|^2=1\}
\)
and let $g_1$ be the round metric of sectional curvature one.  The vector
field
\(
 V(z,w)=(iz,iw)
\)
is tangent to the fibers of the Hopf fibration and has unit length with
respect to $g_1$.  For $\alpha>0$, define
\begin{equation*}\label{eq:berger-metric}
 g_\alpha(X,Y)=g_1(X,Y)+(\alpha-1)g_1(X,V)g_1(Y,V).
\end{equation*}
Thus $g_\alpha(V,V)=\alpha$, and
\begin{equation*}\label{eq:xi}
 \xi=\frac{V}{\sqrt\alpha}
\end{equation*}
is the unit vertical Killing field.  We denote $(\Sph^3,g_\alpha)$ by
$\Sph^3_\alpha$.

Let $\varphi:\Sigma^2\to\Sph^3_\alpha$ be a two-sided immersion with unit
normal $N$, and define the angle function
\begin{equation*}\label{eq:angle}
 \nu=\langle N,\xi\rangle.
\end{equation*}
The sectional curvature of the tangent plane and the Ricci curvature in
the normal direction are{,} respectively{,}
\begin{equation}\label{eq:ambient-ricci}
 \overline K_\Sigma
 =\alpha+4(1-\alpha)\nu^2      \quad \text{ and } \quad
 \Ric_\alpha(N,N)
 =4-2\alpha-4(1-\alpha)\nu^2.
 \end{equation}
These formulas follow, for example, from the curvature tensor recorded in 
\cite[eq.\ (2.8) and (2.9)]{TorralboUrbano2022}.

For later use, we indicate how the dependence on $\nu$ arises.  If
$X\wedge Y$ is an orthonormal horizontal plane, its sectional curvature is
$4-3\alpha$, whereas a plane containing $\xi$ has sectional curvature~$\alpha$.  Since the squared area of the horizontal component of the
oriented tangent plane is $\nu^2$, the curvature of the tangent plane is
\[
 (4-3\alpha)\nu^2+\alpha(1-\nu^2)
 =\alpha+4(1-\alpha)\nu^2.
\]
Taking the trace of the ambient curvature operator over an orthonormal
basis of $N^\perp$ gives \eqref{eq:ambient-ricci}.  This also checks that,
when $\alpha=1$, both formulas reduce to those of the unit round sphere.

We use the convention $H=\tfrac12\operatorname{tr}\sigma$.  Since
$\det A=2H^2-\tfrac12|\sigma|^2$, the Gauss equation gives
\begin{equation}\label{eq:gauss}
 K=2H^2-\frac{|\sigma|^2}{2}
   +\alpha+4(1-\alpha)\nu^2.
\end{equation}
Combining \eqref{eq:ambient-ricci} and \eqref{eq:gauss}, we obtain the
identity
\begin{equation}\label{eq:potential}
 |\sigma|^2+\Ric_\alpha(N,N)
 =-2K+4(H^2+1)+4(1-\alpha)\nu^2.
\end{equation}

\section{The auxiliary spherical immersion}\label{sec:auxiliary}

Throughout this section, $0<\alpha<1$.  Let $\CP^2(1-\alpha)$ denote the
complex projective plane with holomorphic sectional curvature
$4(1-\alpha)$.  There is an isometric {embedding}
\begin{equation*}\label{eq:Falpha}
F_\alpha:\Sph^3_\alpha\to\CP^2(1-\alpha)
\end{equation*}
whose image is a geodesic sphere
\cite[Proposition~2.1]{TorralboUrbano2012}.  If $\eta$ is its unit normal and
$\widehat\sigma$ its second fundamental form, then
\begin{equation}\label{eq:Falpha-second-form}
 -\langle\widehat\sigma(v,w),\eta\rangle
 =\sqrt\alpha\mskip2mu\langle v,w\rangle
 +\frac{\alpha-1}{\sqrt\alpha}
  \langle v,\xi\rangle\langle w,\xi\rangle.
\end{equation}

The formula has two principal curvatures: $\sqrt\alpha$ on the horizontal
distribution and $(2\alpha-1)/\sqrt\alpha$ in the Hopf direction.  It is
therefore enough to verify \eqref{eq:Falpha-second-form} on a horizontal
orthonormal pair together with $\xi$; polarization then gives the stated
bilinear identity.  This is precisely the normalization in
\cite[eq. (2.2)]{TorralboUrbano2012}.  Notice that both the metric of
$\CP^2(1-\alpha)$ and the second fundamental form depend on the same scale.

Let
\(
 \Psi_0:\CP^2(1-\alpha)\to\mathbb R^8
\)
be the first standard isometric embedding, with second fundamental form
$\overline\sigma$.  Its extrinsic geometry is described by
\begin{align}
 \langle\overline\sigma(x,y),\overline\sigma(v,w)\rangle
 ={}&2(1-\alpha)\langle x,y\rangle\langle v,w\rangle+(1-\alpha)\bigl(
 \langle x,w\rangle\langle y,v\rangle
 +\langle x,v\rangle\langle y,w\rangle\notag\\
 &
 +\langle x,Jw\rangle\langle y,Jv\rangle
 +\langle x,Jv\rangle\langle y,Jw\rangle
 \bigr),                                      \label{eq:standard-B}
\end{align}
where $J$ is the complex structure; see
\cite{Ros1984} or \cite[eq. (2.3)]{TorralboUrbano2012}.

The next observation provides the key ingredient for the conformal balancing argument.

\begin{lemma}\label{lem:spherical-image}
The image of $\Psi_0$ is contained in the round sphere
\begin{equation*}\label{eq:spherical-target}
 \Sph^7(R_\alpha)\subset\mathbb R^8,
 \qquad R_\alpha=\frac1{\sqrt{3(1-\alpha)}}.
\end{equation*}
\end{lemma}

\begin{proof}
Identify $\mathbb R^8$ with
\(
 \mathcal H_0(3)=
 \{A\in\operatorname{Herm}(3):\operatorname{tr}A=0\},
 \) endowed with \(\langle A,B\rangle=\operatorname{tr}(AB).
\)
For $[z]\in\CP^2$, let
\(
P_z=\frac{zz^*}{|z|^2},
\)
where \(z^*=\overline{z}^{\mskip2mu T}\) denotes the conjugate transpose of the
column vector \(z\in\mathbb C^3\setminus\{0\}\). Then \(P_z\) is the
Hermitian orthogonal projection onto the complex line spanned by \(z\);
in particular,
\(
P_z^2=P_z
\) {and} \(
\operatorname{tr}P_z=1.
\)
With this notation, the chosen normalization of the standard embedding is
\begin{equation*}\label{eq:standard-explicit}
 \Psi_0([z])=
 \frac1{\sqrt{2(1-\alpha)}}
 \left(P_z-\frac13I\right).
\end{equation*}
Since
\(
 \operatorname{tr}\bigl(P_z-\frac13I\bigr)^2=\frac23,
\)
we have $|\Psi_0([z])|^2=1/[3(1-\alpha)]$, proving the claim.
\end{proof}

\begin{lemma}\label{lem:kahler-angle}
Let $\{e_1,e_2\}$ be a positively oriented orthonormal basis of
$T_p\Sigma$.  After choosing the compatible orientations of the geodesic
sphere and of $\Sigma$, one has
\begin{equation*}\label{eq:kahler-angle}
 \langle Je_1,e_2\rangle=\nu.
\end{equation*}
In particular, the square of the K\"ahler function of the surface in
$\CP^2(1-\alpha)$ is $\nu^2$ and is independent of these orientation choices.
\end{lemma}

\begin{proof}
Let \(\xi_\alpha\) denote the unit vertical vector field of the Hopf
fibration of \(\mathbb S_\alpha^3\). A distinguished feature of the
isometric immersion
\(
 F_\alpha:\mathbb S_\alpha^3
 \to \mathbb{CP}^2(1-\alpha)
\)
is that the Hopf direction is mapped onto the structure direction of the
real hypersurface
\(
 M^3=F_\alpha(\mathbb S_\alpha^3)
 \subset\mathbb{CP}^2(1-\alpha),
\)
that is, if \(\eta\) is a suitable choice of unit normal to
\(M^3\), then
\(
 dF_\alpha(\xi_\alpha)=-J\eta,
\)
where \(J\) denotes the complex structure of
\(\mathbb{CP}^2(1-\alpha)\).

Thus, after identifying
\(\mathbb S_\alpha^3\) with its image under \(F_\alpha\), we use the same
notation \(\xi\) for the unit Hopf vector field and for the structure
vector field of \(M^3\), that is,
\(
 \xi=-J\eta.
\)
Changing the orientation of \(\eta\) changes the sign of this identity,
but does not affect any of the squared quantities used below.

Now let
\(
 \varphi:\Sigma^2\to\mathbb S_\alpha^3
\)
be an oriented immersed surface. Through the composition
\(F_\alpha\circ\varphi\), we regard \(\Sigma\) as a surface contained in
the real hypersurface
\[
 \Sigma^2\subset M^3=F_\alpha(\mathbb S_\alpha^3)
 \subset\mathbb{CP}^2(1-\alpha).
\]
Let \(N\) be the unit normal of \(\Sigma\) in \(M^3\). Along \(\Sigma\),
decompose the structure vector field as
\[
 \xi=\xi^\top+\nu N,
 \qquad
 \nu=\langle\xi,N\rangle.
\]
Choose a local oriented orthonormal frame \(\{e_1,e_2\}\) of
\(T\Sigma\) such that \(e_1\perp\xi^\top\). Thus, after possibly changing
the sign of \(e_2\), we may write
\[
 \xi=a e_2+\nu N,
 \qquad a^2+\nu^2=1.
\]
Using \(J\eta=-\xi\), we obtain
\[
 \langle Je_1,\eta\rangle
 =-\langle e_1,J\eta\rangle
 =\langle e_1,\xi\rangle
 =0.
\]
Therefore, as \(
 \langle Je_1,e_1\rangle=0
\), we deduce that
\(
 Je_1\in\operatorname{span}\{e_2,N\}.
\)
After choosing the orientations of
\(\{e_1,e_2,N,\eta\}\) compatibly with the Hermitian orientation of
\(\mathbb{CP}^2(1-\alpha)\), one has
\[
 Je_1=\nu e_2-aN,
\]
and so
\(
 \langle Je_1,e_2\rangle
 =\nu.
\)
\end{proof}

Consider the composed Euclidean immersion
\begin{equation}\label{eq:composition}
 x=\Psi_0\circ F_\alpha\circ\varphi:
 \Sigma\to\Sph^7(R_\alpha)\subset\mathbb R^8,
\end{equation}
and denote its Euclidean mean curvature vector by $H_x$.

\begin{proposition}\label{prop:Hx}
For the immersion \eqref{eq:composition},
\begin{equation}\label{eq:Hx}
 |H_x|^2
 =H^2+
 \left[\sqrt\alpha+
 \frac{\alpha-1}{2\sqrt\alpha}(1-\nu^2)\right]^2
 +(1-\alpha)(3+\nu^2).
\end{equation}
\end{proposition}

\begin{proof}
Let $\{e_1,e_2\}$ be a local orthonormal tangent frame on $\Sigma$.  Since
\[
 \sum_{i=1}^2\langle e_i,\xi\rangle^2
 =|\xi^\top|^2=1-\nu^2,
\]
formula \eqref{eq:Falpha-second-form} gives
\[
 \frac12\sum_{i=1}^2\widehat\sigma(e_i,e_i)
 =-\left[\sqrt\alpha+
 \frac{\alpha-1}{2\sqrt\alpha}(1-\nu^2)\right]\eta.
\]
Thus the mean curvature vector of $F_\alpha\circ\varphi$ in $\CP^2$ is
the orthogonal sum of $HN$ and the preceding vector. By Lemma~\ref{lem:kahler-angle}, its K\"ahler function satisfies
$\langle Je_1,e_2\rangle^2=\nu^2$.  Formula \eqref{eq:standard-B} then yields
\[
 \left|\frac12\sum_{i=1}^2
 \overline\sigma(e_i,e_i)\right|^2
 =(1-\alpha)(3+\nu^2).
\]
Indeed, this follows by applying \eqref{eq:standard-B} to
$(e_i,e_i,e_j,e_j)$ and summing over $i,j$.  Explicitly,
\[
 \frac14\sum_{i,j=1}^2
 \langle\overline\sigma(e_i,e_i),
 \overline\sigma(e_j,e_j)\rangle
 =(1-\alpha)\bigl(3+\langle Je_1,e_2\rangle^2\bigr).
\]
The three contributions belong to mutually orthogonal normal subspaces,
and \eqref{eq:Hx} follows.
\end{proof}

\section{Weighted balancing and conformal energy}\label{sec:conformal}

We record the version of the Hersch--Li--Yau argument required below.  The
weight will later be a positive first eigenfunction of $-\mathcal J$.

\begin{proposition}\label{prop:balancing}
Let $x:\Sigma\to\Sph^m(R)\subset\mathbb R^{m+1}$ be an immersion of a
closed surface, and let $u>0$ be continuous.  There exists
$\Gamma\in\Conf(\Sph^m)$ such that
\begin{equation*}\label{eq:Phi}
 \Phi=\Gamma\circ\frac{x}{R}:\Sigma\to\Sph^m
\end{equation*}
satisfies
\begin{equation}\label{eq:balancing-general}
 \int_\Sigma u\mskip2mu\Phi\mskip2mu dA=0
 \qquad\text{in }\mathbb R^{m+1}.
\end{equation}
If $H_x$ is the Euclidean mean curvature vector of $x$, then
\begin{equation}\label{eq:energy-general}
 \int_\Sigma|d\Phi|^2\mskip1mu dA
 \leq2\int_\Sigma|H_x|^2\mskip1mu dA.
\end{equation}
\end{proposition}

\begin{proof}
For $a$ in the open unit ball $B^{m+1}$, take the M\"obius transformation
\[
 \Gamma_a(y)=
 \frac{(1-|a|^2)y+2(1+\langle a,y\rangle)a}
 {1+2\langle a,y\rangle+|a|^2},
 \qquad y\in\Sph^m.
\]
It maps $\Sph^m$ conformally onto itself.  The map
\[
 \mathcal F(a)=\frac1{\int_\Sigma u\mskip2mu dA}
 \int_\Sigma u\mskip2mu\Gamma_a(x/R)\mskip2mu dA
\]
is continuous.  As $a\longrightarrow b\in\partial B^{m+1}$, the maps $\Gamma_a$
converge to the constant map $b$ away from one point.  Since $u\mskip2mu dA$ has no
atoms, dominated convergence shows that $\mathcal F(a)\longrightarrow b$.  Hence
$\mathcal F$ extends to the closed ball and equals the identity on its
boundary.  If $0$ were not in its image, radial projection would produce a
retraction of the closed ball onto its boundary, a contradiction.  Thus
there is $a_0\in B^{m+1}$ such that $\mathcal F(a_0)=0$.  This is the
weighted Hersch center-of-mass argument \cite{Hersch1970,LiYau1982}.  Taking
$\Gamma=\Gamma_{a_0}$ proves \eqref{eq:balancing-general}.

The dilation $x\longmapsto x/R$ and $\Gamma$ are restrictions of conformal
transformations of $\mathbb R^{m+1}\cup\{\infty\}$.  Because $\Phi$ is a
conformal immersion into the unit sphere,
\[
 \frac12\int_\Sigma|d\Phi|^2\mskip1mu dA=\Area(\Phi),
\]
with multiplicity.  The Euclidean Willmore functional of a closed surface
is invariant under this transformation.  For a dilation, the scaling of
$|H|^2$ cancels that of the area element; inversion invariance gives the
nontrivial M\"obius step.  Therefore
\[
 \int_\Sigma|H_x|^2\mskip1mu dA
 =\int_\Sigma|H^{\mathbb R^{m+1}}_\Phi|^2\mskip1mu dA_\Phi.
\]
For a surface in the unit sphere,
\(
 |H^{\mathbb R^{m+1}}_\Phi|^2
 =1+|H^{\Sph^m}_\Phi|^2\geq1.
\)
Hence,
\[
 \Area(\Phi)
 \leq\int_\Sigma|H_x|^2\mskip1mu dA,
\]
which is equivalent to \eqref{eq:energy-general}.
\end{proof}

\section{The second eigenvalue estimate}\label{sec:estimate}

Next, we derive a sharp estimate for the second eigenvalue of the Jacobi operator. The result reads as follows.

\begin{theorem}\label{thm:main}
Let $0<\alpha<1$, and let
$\varphi:\Sigma^2\to\Sph^3_\alpha$ be a closed two-sided immersion.  Then
\begin{align}
 \lambda_2(\mathcal J)\leq{}&
 \frac{(1+\alpha)(1-3\alpha)}{2\alpha}
 -\frac{2}{\Area(\Sigma)}\int_\Sigma H^2\mskip1mu dA
 +\frac{4\pi\chi(\Sigma)}{\Area(\Sigma)}\notag\\
 &-\frac{(1-\alpha)^2}{2\alpha\Area(\Sigma)}
 \int_\Sigma(2-\nu^2)\nu^2\mskip1mu dA.                 \label{eq:main}
\end{align}
In particular,
\begin{equation}\label{eq:coarse}
 \lambda_2(\mathcal J)\leq
 \frac{(1+\alpha)(1-3\alpha)}{2\alpha}
 -\frac{2}{\Area(\Sigma)}\int_\Sigma H^2\mskip1mu dA
 +\frac{4\pi\chi(\Sigma)}{\Area(\Sigma)}.
\end{equation}
\end{theorem}

\begin{proof}
Let $u_1>0$ be a first eigenfunction of $-\mathcal J$.  By
Lemma~\ref{lem:spherical-image}, the immersion $x$ in
\eqref{eq:composition} takes values in $\Sph^7(R_\alpha)$.  Apply
Proposition~\ref{prop:balancing} with weight $u_1$.  We obtain
\(
 \Phi=(\Phi_1,\ldots,\Phi_8):\Sigma\longrightarrow\Sph^7
\)
such that
\begin{equation*}\label{eq:coordinate-orthogonality}
 \int_\Sigma u_1\Phi_A\mskip1mu dA=0,
 \qquad A=1,\ldots,8.
\end{equation*}
Thus every coordinate $\Phi_A$ is admissible in the min--max
characterization of $\lambda_2(\mathcal J)$.  Therefore
\[
 \lambda_2(\mathcal J)\int_\Sigma\Phi_A^2\mskip1mu dA
 \leq\int_\Sigma|\nabla\Phi_A|^2\mskip1mu dA
 -\int_\Sigma
 \bigl(|\sigma|^2+\Ric_\alpha(N,N)\bigr)\Phi_A^2\mskip1mu dA.
\]
Summing over $A$ and using $\sum_A\Phi_A^2=1$, we get
\begin{equation}\label{eq:rayleigh}
 \lambda_2(\mathcal J)\Area(\Sigma)
 \leq\int_\Sigma|d\Phi|^2\mskip2mu dA
 -\int_\Sigma\bigl(|\sigma|^2+\Ric_\alpha(N,N)\bigr)\mskip2mu dA.
\end{equation}

By Proposition~\ref{prop:balancing} and Proposition~\ref{prop:Hx},
\begin{equation*}\label{eq:energy-bound}
 \int_\Sigma|d\Phi|^2\mskip1mu dA
 \leq2\int_\Sigma|H_x|^2\mskip1mu dA.
\end{equation*}
Substitution of \eqref{eq:Hx} and \eqref{eq:potential} into
\eqref{eq:rayleigh} gives
\begin{align*}
 \lambda_2(\mathcal J)\Area(\Sigma)
 \leq\int_\Sigma\biggl\{
 -2H^2+2K
 &+2\left[\sqrt\alpha+
 \frac{\alpha-1}{2\sqrt\alpha}(1-\nu^2)\right]^2+6(1-\alpha)-4-2(1-\alpha)\nu^2
 \biggr\}\mskip2mu dA.
\end{align*}
The remaining calculation is the pointwise algebraic identity
\begin{align}
 &2\left[\sqrt\alpha+
 \frac{\alpha-1}{2\sqrt\alpha}(1-\nu^2)\right]^2
 +6(1-\alpha)-4-2(1-\alpha)\nu^2\notag\\
 &\qquad=
 \frac{(1+\alpha)(1-3\alpha)}{2\alpha}
 -\frac{(1-\alpha)^2}{2\alpha}(2-\nu^2)\nu^2. \notag   \label{eq:algebra}
\end{align}
Hence
\begin{align*}
 \lambda_2(\mathcal J)\Area(\Sigma)\leq{}&
 \frac{(1+\alpha)(1-3\alpha)}{2\alpha}\Area(\Sigma)
 -2\int_\Sigma H^2\mskip1mu dA+2\int_\Sigma K\mskip2mu dA\\
 &-\frac{(1-\alpha)^2}{2\alpha}
 \int_\Sigma(2-\nu^2)\nu^2\mskip2mu dA.
\end{align*}
Gauss--Bonnet now proves \eqref{eq:main}.  Since
{$2-\nu^2\geq0$}, dropping the last term gives
\eqref{eq:coarse}.
\end{proof}

For the sake of completeness, we present the proof of the following result due to the second author \cite{Mendes2019}, for the round sphere.

\begin{proposition}\label{prop:round}
For a closed two-sided immersion $\Sigma^2\to\Sph^3_1$, we have
\[
 \lambda_2(\mathcal J)\leq
 -2-\frac{2}{\Area(\Sigma)}\int_\Sigma H^2\mskip1mu dA
 +\frac{4\pi\chi(\Sigma)}{\Area(\Sigma)}.
\]
\end{proposition}

\begin{proof}
Apply Proposition~\ref{prop:balancing} directly to the standard inclusion
$\Sigma\to\Sph^3\subset\mathbb R^4$.  Its Euclidean mean curvature vector
satisfies $|H^{\mathbb R^4}|^2=1+H^2$.  Repeating
\eqref{eq:rayleigh} and using
$|\sigma|^2+2=-2K+4H^2+4$ yields the stated estimate.  It also agrees with
the limit of \eqref{eq:main} as $\alpha\to1^-$.  Notice, however, that the
projective realization itself degenerates at $\alpha=1$; the endpoint is
proved directly rather than by inserting $\alpha=1$ into that realization.
\end{proof}

\section{Main result}\label{sec:rigidity}

We first extract the sign consequence of the refined estimate.

\begin{corollary}\label{cor:sign}
Let $1/3\leq\alpha\leq1$, and let
$\varphi:\Sigma^2\to\Sph^3_\alpha$ be a closed two-sided immersion with
$\chi(\Sigma)\leq0$.  Then $\lambda_2(\mathcal J)\leq0$.  If
$1/3<\alpha\leq1$, then $\lambda_2(\mathcal J)<0$.
\end{corollary}

\begin{proof}
For $1/3\leq\alpha<1$,
\[
 \frac{(1+\alpha)(1-3\alpha)}{2\alpha}\leq0,
\]
and all the other terms in \eqref{eq:main} are nonpositive.  The displayed
constant is strictly negative if $\alpha>1/3$.  At $\alpha=1$, use
Proposition~\ref{prop:round}.
\end{proof}

\subsection{The spectrum of the Clifford torus}

It is useful to compare the constant appearing in Theorem~\ref{thm:intro-main} with the
exact second Jacobi eigenvalue of the Clifford torus. This comparison
also explains the distinguished role of the critical parameter
$\alpha=1/3$ in our estimate.

Let
\[
 \mathcal{T}=
 \left\{(z,w)\in\Sph^3:
 |z|=|w|=\frac1{\sqrt2}\right\}
 \subset\Sph^3_\alpha
\]
be the Clifford torus, parametrized by
\[
 F(\theta,\varphi)
 =
 \frac1{\sqrt2}\bigl(e^{i\theta},e^{i\varphi}\bigr).
\]
The metric induced by $g_\alpha$ is
\[
 g_\mathcal{T}
 =
 \frac14\left[
 (1+\alpha)(d\theta^2+d\varphi^2)
 +2(\alpha-1)\mskip1mu d\theta\mskip1mu d\varphi
 \right].
\]
Consequently, the Fourier mode
$e^{i(m\theta+n\varphi)}$ is an eigenfunction of $-\Delta_\mathcal{T}$ with
eigenvalue
\begin{equation}\label{eq:flat-spectrum}
 \mu_{m,n}
 =
 \frac{(m+n)^2}{\alpha}+(m-n)^2,
 \qquad (m,n)\in\mathbb Z^2.
\end{equation}

The Clifford torus is minimal for every Berger metric, the Hopf field
is tangent to it, and its induced metric is flat. Thus
\[
 H=0,\qquad \nu=0,\qquad K=0.
\]
Moreover, the Gauss equation and the ambient Ricci-curvature formula
give
\(
 |\sigma|^2+\Ric_\alpha(N,N)=4.
\)
Therefore,
\(
 \mathcal{J}_\mathcal{T}=\Delta_\mathcal{T}+4,
\)
and the eigenvalues of $-\mathcal{J}_\mathcal{T}$ are
\[
 \lambda_{m,n}
 =
 \frac{(m+n)^2}{\alpha}+(m-n)^2-4.
\]

The constant functions yield the simple first eigenvalue $-4$. For
$0<\alpha\le1$, the smallest positive eigenvalue of $-\Delta_\mathcal{T}$ is
\[
 \min\left\{4,1+\frac1\alpha\right\}.
\]
Indeed, the value $4$ is attained at $(m,n)=(1,-1)$ and $(-1,1)$,
whereas $1+\alpha^{-1}$ is attained at
\[
 (m,n)\in\{(\pm1,0),(0,\pm1)\}.
\]
It follows that
\begin{equation}\label{eq:exact-second-clifford}
 \lambda_2(\mathcal{J}_\mathcal{T})
 =
 \begin{cases}
  0,
  & 0<\alpha\le\frac13,\\[5pt]
  \displaystyle\frac{1-3\alpha}{\alpha},
  & \frac13\le\alpha\le1.
 \end{cases}
\end{equation}
In particular,
\begin{equation}\label{eq:model-lambda2}
 \lambda_2(\mathcal{J}_\mathcal{T})=0
 \qquad\text{on }\mathcal{T}\subset\Sph^3_{1/3}.
\end{equation}

We now compare this exact value with the estimate of Theorem~\ref{thm:intro-main}.
Since $H=0$, $\chi(\mathcal{T})=0$, and $\nu=0$, the right-hand side of that
estimate, evaluated on $\mathcal{T}$, reduces to
\[
 C_\alpha
 =
 \frac{(1+\alpha)(1-3\alpha)}{2\alpha}.
\]
For $1/3\le\alpha\le1$, we have
\[
 C_\alpha-\lambda_2(\mathcal{J}_\mathcal{T})
 =
 \frac{(1-\alpha)(3\alpha-1)}{2\alpha}.
\]
Hence the estimate of Theorem~\ref{thm:intro-main}, together with its round counterpart
in Proposition~\ref{prop:round}, is an equality on the Clifford torus precisely at
the critical parameter $\alpha=1/3$ and at the round endpoint
$\alpha=1$. For $1/3<\alpha<1$, the estimate is strict even on the
Clifford torus. These computations explain why we focus on the case of vanishing second eigenvalue. As a by-product, our approach also provides a framework for classifying stable or index-one surfaces, depending on the geometric setting, in the spirit of \cite{TorralboUrbano2012}.

\medskip

\subsection{Sharpness and rigidity}
We next recall the elementary global consequence of the condition $\nu=0$. For more details see the discussion of Hopf tori in
\cite{TorralboUrbano2012}.

\begin{lemma}\label{lem:Hopf-surface}
Let $\varphi:\Sigma^2\to\Sph^3_\alpha$ be a closed connected immersion.
If $\nu=0$, then its image is a Hopf torus $\pi^{-1}(\gamma)$, where
$\pi:\Sph^3_\alpha\to\Sph^2$ is the Hopf fibration and $\gamma$ is a
closed immersed curve in the base.  Moreover,
\begin{equation}\label{eq:Hopf-H}
 2H=k_g
\end{equation}
up to the simultaneous choice of orientations, where $k_g$ is the geodesic
curvature of $\gamma$ in the base metric.
\end{lemma}

\begin{proof}
The identity $\nu=\langle N,\xi\rangle=0$ says that $\xi$ is tangent to the
surface.  Its integral curves are the Hopf fibers, so the image is saturated
by fibers.  The quotient is therefore a closed immersed curve $\gamma$ in
the base and the image is $\pi^{-1}(\gamma)$.  Computing the trace of the
second fundamental form in the orthonormal frame formed by the vertical
direction and the horizontal lift of the tangent to $\gamma$ gives
\eqref{eq:Hopf-H}.  Compactness implies that the immersion factors
through a finite covering of this Hopf torus.
\end{proof}

In the following result, we establish a rigidity theorem for the critical Berger metric.

\begin{theorem}\label{thm:rigidity}
Let
\(
 \varphi:\Sigma^2\longrightarrow\Sph^3_{1/3}
\)
be a closed, connected, two-sided immersion with $\chi(\Sigma)\leq0$.
Then
\[
 \lambda_2(\mathcal J)\leq0.
\]
If equality holds, then, up to an ambient isometry, the image of
$\varphi$ is the Clifford torus and $\varphi$ factors through a finite
covering of its standard embedding.  Conversely, the standard
Clifford-torus embedding {realizes} equality.  Consequently, in the embedded
category equality holds if and only if the surface is congruent to the
Clifford torus.
\end{theorem}

\begin{proof}
At $\alpha=1/3$, estimate \eqref{eq:main} becomes
\begin{equation*}\label{eq:critical}
 \lambda_2(\mathcal J)\leq
 -\frac{2}{\Area(\Sigma)}\int_\Sigma H^2\mskip1mu dA
 +\frac{4\pi\chi(\Sigma)}{\Area(\Sigma)}
 -\frac{2}{3\Area(\Sigma)}
 \int_\Sigma(2-\nu^2)\nu^2\mskip1mu dA.
\end{equation*}
Every term on the right-hand side is nonpositive.  If
$\lambda_2(\mathcal J)=0$, then the right-hand side must vanish.  Hence
\[
 H\equiv0,\qquad \chi(\Sigma)=0,\qquad \nu\equiv0.
\]
By Lemma~\ref{lem:Hopf-surface}, the image is the inverse image of a closed
curve $\gamma$ in the base and $2H=k_g$.  Therefore $k_g=0$, so $\gamma$
is a geodesic.  The isometry group of the base acts transitively on its
geodesics, and the inverse image of any such geodesic is congruent to the
Clifford torus.  The covering statement follows from the final assertion of
Lemma~\ref{lem:Hopf-surface}.

Conversely, \eqref{eq:model-lambda2} shows that the standard Clifford torus
in $\Sph^3_{1/3}$ satisfies $\lambda_2(\mathcal J)=0$.  If $\varphi$ is {an embedding},
the covering obtained above has degree one, which completes the equality
characterization in the embedded category.  
\end{proof}

\begin{remark}
For a CMC surface, weak stability means that the quadratic form associated
with $-\mathcal J$ is nonnegative on functions of integral zero.  This is closely
related to, but not identical with, the sign of $\lambda_2(\mathcal J)$, whose
min--max space is the orthogonal complement of a first eigenfunction.
Torralbo and Urbano used CMC stability and specially chosen mean-zero test
functions to classify stable examples.  Here the weighted balancing uses
the first eigenfunction itself and therefore targets $\lambda_2(\mathcal J)$
directly.  This distinction is what permits Theorems~\ref{thm:main} and
\ref{thm:rigidity} to apply without the CMC hypothesis.
\end{remark}

\section*{Funding}
The authors were partially supported by the Brazilian National Council for Scientific and Technological Development, Brazil [Grants: 402563/2023-9 and 304381/2026-8 to M.B.; 309867/2023-1 and 445723/2025-4 to A.M.], and were partially supported by Coordination for the Improvement of Higher Education Personnel [Finance code - 001].

\section*{Data Availability}
Data availability is not applicable to this article, as no data sets were generated or analyzed in the course of this research.

\section*{Conflict of Interest}
The authors declare that they have no conflict of interest related to this article.

\bibliographystyle{amsplain}
\bibliography{references.bib}

\end{document}